\documentclass[a4paper,10pt]{amsart}
\usepackage{amssymb,amsmath}
\newcommand{\Q}{\mathbb{Q}}
\newcommand{\R}{\mathbb{R}}

\newcommand{\N}{\mathbb{N}}
\newtheorem{theorem}{Theorem}
\newtheorem{lemma}{Lemma}

\newcommand{\Z}{\mathbb{Z}}

\newcommand{\twosum}[2]{\sum_{\substack{#1\\#2}}}

\newcommand{\ep}{\varepsilon}

\newcommand{\Mod}[1]{\;(\operatorname{mod}\,#1)}
\makeatletter
\@namedef{subjclassname@2020}{\textup{2020} Mathematics Subject Classification}
\makeatother

\begin{document}
\title{Bounds for the Quartic Weyl Sum}
\author{D.R. Heath-Brown\\Mathematical Institute, Oxford}
\date{}

\begin{abstract}

We improve the standard Weyl estimate 
for quartic exponential sums in which the argument is a quadratic
irrational.  Specifically we show that
\[\sum_{n\le N} e(\alpha n^4)\ll_{\ep,\alpha}N^{5/6+\ep}\]
for any $\ep>0$ and any quadratic irrational $\alpha\in\R-\Q$.
Classically one would have had the exponent $7/8+\ep$
for such $\alpha$. In contrast to the author's earlier work \cite{cubweyl}
on cubic 
Weyl sums (which was conditional on the $abc$-conjecture), we show that the 
van der Corput $AB$-steps are sufficient for the quartic case, rather than the 
$BAAB$-process needed for the cubic sum.
\end{abstract}

\address{Mathematical Institute\\
Radcliffe Observatory Quarter\\ Woodstock Road\\ Oxford\\ OX2 6GG}
\email{rhb@maths.ox.ac.uk}


\maketitle

\section{Introduction}
This paper is concerned with estimates for the Weyl sum
\[S_k(\alpha,N)=\sum_{n\le N} e(\alpha n^k),\]
in the particular case $k=4$. Here  $e(x)=\exp(2\pi i x)$ as usual,
and $k\ge 3$ is an integer. In a companion paper \cite{cubweyl}
we examined the cubic Weyl sum $S_3(\alpha,N)$, but there are significant
differences between the two investigations, as well as similarities. 

The classical ``Weyl bound'' takes the form
\begin{equation}\label{WB}
S_k(\alpha,N)\ll_{\ep}N^{1-2^{1-k}+\ep},
\end{equation}
where $\ep$ is an arbitrary small
positive number, whenever  $\alpha$ has a rational approximation
$a/q$ such that 
\begin{equation}\label{WBC}
|\alpha-a/q|\ll \frac{1}{qN^{k-1}},\;\;\; N\ll q\ll N^{k-1}.
\end{equation}
This is essentially Lemma 3 in Hardy and Littlewood's 
first `Partitio Numerorum' paper \cite{HL}. The method however is 
due to Weyl \cite {Weyl}, who did not work out a quantitative estimate.

If $\alpha\in\R-\Q$ is algebraic one may conclude from Roth's theorem
that there is always a fraction $a/q$ satisfying (\ref{WBC}) when $N$
is large enough. In the case $k=4$ we deduce that   
\begin{equation}\label{WB4}
S_4(\alpha,N)\ll_{\ep,\alpha}N^{7/8+\ep}
\end{equation}
for such $\alpha$.

The general exponent $7/8$ in (\ref{WB4}) has never been improved on,
but our goal in 
the present paper is to show that one can do better for special values
of $\alpha$.  We shall prove the following bound.

\begin{theorem}\label{sqrt2}
Let $\alpha\in\R-\Q$ be a quadratic irrational.  Then
\[S_4(\alpha,N)\ll_{\ep,\alpha}N^{5/6+\ep}\]
for any $\ep>0$.
\end{theorem}

We should compare this with the principal result from the author's
previous work \cite{cubweyl}, in which the Weyl exponent $3/4+\ep$ for
cubic exponential sums was reduced to $5/7+\ep$ for quadratic
irrational $\alpha$, but only under the assumption of the
$abc$-conjecture. For the quartic sum we are able to establish an
unconditional result. Moreover, in the quartic case we have reduced
the exponent by $7/8-5/6=1/24$, while the corresponding reduction in
the cubic case was only $3/4-5/7=1/28$. Thus our new result achieves a
better saving, without any unproved hypothesis.

As with the previous paper we use the $q$-analogue of van der Corput's
method. This requires a good approximation $a/q$ to $\alpha$,
in which $q$ factorizes in a suitable way.  For the quartic sum we
prove the following.
\begin{theorem}\label{genthm}
  Suppose that $a$ and $q$ are coprime integers and that $q$ factors
  as $q=q_1q_2$.  Then 
\[S_4(\alpha,N)\ll_{\ep}
\left(1+N^4\left|\alpha-\frac{a}{q}\right|\right)
(N^{1/2}q_1^{1/2}+N^{1/2}q_2^{1/4}+Nq_2^{-1/6})q^{\ep},\]
for any $N\ge 2q_1$ and any $\ep>0$.
\end{theorem}

For the cubic sum the argument of \cite{cubweyl} used the van der
Corput $BAAB$ steps. The final $B$-process produced a complete exponential sum
\[\twosum{w,x,y,z\!\!\!\!\!\pmod{q}}{q|w-x-y+z+t}
e\left(\frac{c(w^3-x^3-y^3+z^3)+uq_2(w-x)+mq_1(w-y)}{q}\right)\]
for certain integers $q_1$ and $q_2$. When $q$ is square-free
we could give a suitably good bound for this, but not for general
$q$. In the paper \cite{cubweyl} the $abc$ conjecture allowed one to
restrict attention to square-free $q$.

In contrast, for our Theorem \ref{sqrt2} it suffices to use the much
simpler van der Corput $AB$-process.  This will produce a complete
exponential sum
\[S(a,h;q)=\sum_{n=1}^qe((an^3+hn)/q),\]
for which we have good bounds for all moduli $q$.  This allows
us to give an unconditional treatment of the quartic sum $S_4(\alpha,N)$.

Perhaps the most surprising aspect of this paper is that in handling
$S_k(\alpha,N)$ for $k=4$ rather than $k=3$ one replaces the $BAAB$
process by the simpler $AB$ process.

We remark that one can improve on the factor $1+N^4|\alpha-a/q|$ in
Theorem \ref{genthm} when one has $|\alpha-a/q|\gg N^{-4}$. Indeed,
even when $q$ is not known to factor one can use a standard version of
the van der Corput $AB$-process to obtain non-trivial results. 
However this will not be relevant for the application to Theorem \ref{sqrt2}.

{\bf Acknowledgement.} The author would like to thank the anonymous
referee for their careful reading of the paper, which has resulted in
the elimination of a number of misprints. 

\section{The van der corput $A$-Process}
In this section we use the $q$-analogue of the van der Corput
A-process, which produces the following result, in which we set
\[\Sigma(q;u,v;I)=\sum_{n\in I}e_q(un^3+vn).\]
\begin{lemma}\label{L1}
  Let $\alpha=a/q+\delta$ with $(q,a)=1$ and $q=q_1q_2$, where
  $1\le q_1\le N/2$. Then
\[S_4(\alpha,N)^2\ll(1+|\delta| N^4)^2q_1\left\{N+\sum_{1\le h\le N/(2q_1)}
\max_I |\Sigma(q_2;4ah,4ah^3q_1^2;I)|\right\},\]
where $I$ runs over sub-intervals of $(0,N]$.
\end{lemma}
\begin{proof}
We have
\begin{eqnarray*}
S_4(\alpha,N)&=&S_4(a/q,N)e(\delta N^4)-\int_0^N(2\pi i\delta)4t^3
e(\delta t^4)S_4(a/q,t)dt\\   &\ll& (1+N^4|\delta|)\max_{t\le N}|S_4(a/q,t)|,
\end{eqnarray*}
by partial summation. We now begin the van der corput $A$-process by
setting $H=[N/2q_1]\ge 1$ and writing
\[HS_4(a/q,t)=
\sum_{h\le H}\;\sum_{n: 0<n+2hq_1\le t}e_q\left(a(n+2hq_1)^4\right).\]
If $0<n+2hq_1\le t$ with $0<h\le H$ and $0\le t\le N$ we will have
$-N<n<N$, so that
\[HS_4(a/q,t)=\sum_{|n|< N}
\sum_{\substack{1\le h\le H\\ 0<n+2hq_1\le t}}e_q\left(a(n+2hq_1)^4\right).\]
Cauchy's inequality then yields
\[H^2|S_4(a/q,t)|^2\le 2N\sum_{|n|< N}
\left|\sum_{\substack{1\le h\le H\\ 0<n+2hq_1\le t}}
e_q\left(a(n+2hq_1)^4\right)\right|^2.\]
We now expand the square, and change the order of summation to produce
\[H^2|S_4(a/q,t)|^2\le 2N\sum_{0<h_1,h_2\le H} \;
\sum_{\substack{n\\ 0<n+2h_1q_1\le t\\ 0<n+2h_2q_1\le t}}
e_q\left(a\{(n+2h_1q_1)^4-(n+2h_2q_1)^4\}\right).\]
We then write $m=n+(h_1+h_2)q_1$ and $h=h_1-h_2$, so that
\[(n+2h_1q_1)^4-(n+2h_2q_1)^4=(m+hq_1)^4-(m-hq_1)^4=4hq_1(m^3+h^2q_1^2m).\]
The conditions $0<n+2h_1q_1\le t$ and $0<n+2h_2q_1\le t$ are
equivalent to the requirement that $|h|q_1<m\le t-|h|q_1$, and we
deduce that
\[H^2|S_4(a/q,t)|^2\le 2N\sum_{0<h_1,h_2\le H} \;\;
\sum_{|h|q_1<m\le t-|h|q_1}e_{q_2}\left(4ah(m^3+h^2q_1^2m)\right).\]
Each value of $h$ arises at most $H$ times, and $|h|\le H$ in each
case, whence
\[H^2|S_4(a/q,t)|^2\le 4NH\sum_{0\le h\le H} \;
\left|\;\sum_{|h|q_1<m\le t-|h|q_1}e_{q_2}\left(4ah(m^3+h^2q_1^2m)\right)\right|.\]
The term $h=0$ produces an overall contribution at most $4N^2H$, and
the lemma follows.
\end{proof}

\section{The van der corput $B$-Process}
In this section we apply the $B$-process to $\Sigma(q_2;4ah,4ah^3q_1^2;I)$.
We begin by writing $d=(q_2,4h)$ and setting $q_2=dr$, $4ah=du$, and
$4ah^3q_1^2=dv$, whence
\begin{equation}\label{B2}
  \Sigma(q_2;4ah,4ah^3q_1^2;I)=\Sigma(r;u,v;I),
  \end{equation}
with $r$ and $u$ coprime.  

For our application, the $B$-process is given by the following lemma.
\begin{lemma}\label{L2}
  When $I$ is a sub-interval of $(0,N]$ we have
  \[\Sigma(r;u,v;I)\ll_\ep (r^{1/2}+Nr^{-1/3})r^\ep\]
for any fixed $\ep>0$.
\end{lemma}
\begin{proof}
  We start from the identity
  \[\Sigma(r;u,v;I)=r^{-1}\sum_{-r/2<m\le r/2}\;\;
  \sum_{s\Mod{r}}e_r(us^3+vs)\sum_{n\in I}e_r\big(m(s-n)\big).\]
  Thus
   \[\Sigma(r;u,v;I)=r^{-1}\sum_{-r/2<m\le r/2}\;\;\sum_{n\in I}e_r(-mn)
   \Sigma(r;u,v+m),\]
   where we define
  \[\Sigma(r;u,v)=\sum_{n\Mod{r}}e_r(un^3+vn).\]   
  However
   \[\sum_{n\in I}e_r(-mn)\ll\min(N,r/|m|),\]
whence
\begin{equation}\label{B1}
\Sigma(r;u,v;I)\ll r^{-1}
  \sum_{|m|\le r/2}\min(N,r/|m|)\left|\Sigma(r;u,v+m)\right|.
  \end{equation}

We now require estimates for $\Sigma(r;u,v)$. Whenever $(r,u)=1$ and
$\ep>0$ we have
\[\Sigma(r;u,v)\ll_\ep r^{2/3+\ep}\]
by Hua \cite{hua2}, and
\[\Sigma(r;u,v)\ll_\ep r^{1/2+\ep}(r,v)\]
by Hua \cite{hua}. To apply these we cover the range $|m|\le r/2$ in (\ref{B1})
with intervals $[m|\le r/N$ and $M<|m|\le 2M$ for suitable dyadic $M$
  with $r/N\ll M\ll r$. We then observe that
  \begin{eqnarray*}
    \sum_{M_0<t\le M_0+M}\min\{r^{1/6},(r,t)\}&\le &
\sum_{d\mid r}\sum_{\substack{M_0<t\le M_0+M\\ d\mid t}}\min\{r^{1/6},d\}\\
&\le &\sum_{d\mid r}(M/d+1)\min\{r^{1/6},d\}\\
&\le &\sum_{d\mid r}(M/d).d + \sum_{d\mid r}1.r^{1/6}\\
&\ll_\ep & r^{\ep}(M+r^{1/6}).
\end{eqnarray*}
Thus
\[\sum_{|m|\le r/N}\min(N,r/|m|)\left|\Sigma(r;u,v+m)\right|
\ll_\ep r^{1/2+\ep}(r/N+r^{1/6})N,\]
and
\[\sum_{M<|m|\le 2M}\min(N,r/|m|)\left|\Sigma(r;u,v+m)\right|
\ll_\ep r^{1/2+\ep}(M+r^{1/6})\frac{r}{M}.\]
Since $r/N\ll M\ll r$ both the above bounds are
$O_\ep(r^{1/2+\ep}(r+Nr^{1/6}))$. We deduce that the sum (\ref{B1}) is
\[\ll_\ep r^{-1}\left\{1+\sum_{\mathrm{dyadic}\; M}1\right\}
r^{1/2+\ep}(r+Nr^{1/6})\ll_\ep (r^{1/2}+Nr^{-1/3})r^{2\ep},\]
and the lemma follows on replacing $\ep$ by $\ep/2$.
\end{proof}

We can now complete the proof of Theorem \ref{genthm}.
\begin{proof}[Proof of Theorem \ref{genthm}]
  In view of the relation (\ref{B2}) we deduce from Lemmas \ref{L1} and
  \ref{L2} that
  \[S_4(\alpha,N)^2\ll_\ep (1+|\delta| N^4)^2q_1
  \left\{N+\sum_{1\le h\le N/(2q_1)}q^\ep(r^{1/2}+Nr^{-1/3})\right\},\]
where $r=q_2/(q_2,4h)$. It follows that
\[S_4(\alpha,N)^2\ll_\ep  (1+|\delta| N^4)^2 q^\ep q_1\left\{N+q_2^{1/2}N/q_1+
Nq_2^{-1/3}\sum_{1\le h\le N/(2q_1)}(q_2,h)^{1/3}\right\}.\]
However a standard argument shows that
\begin{eqnarray*}
  \sum_{1\le h\le N/(2q_1)}(q_2,h)^{1/3}&\le&\sum_{d\mid q_2}d^{1/3}
  \sum_{\substack{h\le N/(2q_1)\\ d\mid h}}1\\
  &\le& \sum_{d\mid q_2}d^{1/3}\frac{N}{2q_1d}\\
  &\ll_\ep& Nq_1^{-1}q^{\ep}.
\end{eqnarray*}
We therefore have
\[S_4(\alpha,N)^2\ll_\ep  (1+|\delta| N^4)^2q^{2\ep}q_1\left\{N+
q_2^{1/2}Nq_1^{-1}+N^2q_2^{-1/3}q_1^{-1}\right\},\]
and the theorem follows.
\end{proof}

\section{Deduction of Theorem \ref{sqrt2}}

The key result which utilizes the special approximation properties of
real quadratic irrationals is the following.

\begin{theorem}\label{smden}
Let $\alpha\in\R$ be a quadratic irrational, and let $\ep>0$ be
given.  Then there is a constant $C(\alpha,\ep)$ such that, for any
$Q\in\N$, one can solve
\[\left|\alpha-\frac{a}{q}\right|\le\frac{C(\alpha,\ep)}{qQ},\;\;\;
(a\in\Z,\;\;q\in\N,\;\; Q\ll_{\ep,\alpha} q\le Q)\]
with $q$ having no prime factors $p>q^{\ep}$.
\end{theorem}
This is essentially Theorem 3 of \cite{cubweyl}. However the assertion
that $q\gg_{\ep,\alpha} Q$ was accidentally omitted from the statement there,
although it is explicitly mentioned in the proof.

To deduce Theorem \ref{sqrt2} we 
approximate $\alpha$ as above, with $Q$ taken to be $N^2$.
We proceed to build a divisor $q_1$ of
$q$, one prime factor at a time, to produce a product in the range
\[q^{1/3}\le q_1\le q^{1/3+\ep}.\]
We will therefore have $q=q_1q_2$ with
\[q^{2/3-\ep}\le q_2\le q^{2/3}.\]
Moreover $q_1\ll _{\ep,\alpha}q^{1/3+\ep}\le N^{2/3+2\ep}$
so that $2q_1\le N$ for large enough $N$.  Since
$N^4|\alpha-a/q|\ll_{\ep,\alpha}1$ 
it now follows from Theorem \ref{genthm} that
\[S_4(\alpha,N)\ll_{\ep,\alpha}
\left(N^{1/2}q^{1/6+\ep/2}+Nq^{-1/9+\ep/6}\right)q^{\ep}
\ll_{\ep,\alpha}N^{5/6+3\ep}.\]
This suffices for Theorem \ref{sqrt2}, on re-defining $\ep$.

\section{Addendum}

Since preparing the original version of this paper it has been pointed
out to the author by Professor Ping Xi that one can handle
$S_k(\alpha,N)$ to good effect by the $q$-analogue of van
der Corput's method for all $k\ge 5$, but subject to the
$abc$-conjecture following the 
author's work \cite{cubweyl}. The general form of the $q$-analogue of
van der Corput's method has been developed by Wu and Xi
\cite{WuXi}. Subject to the $abc$-conjecture one finds that the
$A^sB$-process produces a result for the sum
$S_k(\alpha,N)$, having the same shape as Theorem \ref{sqrt2} but with exponent
\[1-\frac{2s-k}{2(2^s-2)}+\ep.\]
The choice $s=[(k+3)/2]$ is optimal, and produces an improvement on
the Weyl bound for all $k$.

\bigskip
\bigskip


\begin{thebibliography}{9}
\bibitem{HL} G.H. Hardy and J.E. Littlewood, Some problems of
  `Partitio Numerorum'; I; A new solution of Waring's problem, {\em
  Nachr. Akad. Wiss.  
G\"{o}ttingen. Math.-Phys. Kl.},  1920 (1920), 33--54.
\bibitem{cubweyl} D.R. Heath-Brown, Bounds for the cubic Weyl sum,
  {\em J. Math. Sciences}, 171 (2010), 813--823.
  \bibitem{hua2}L.-K. Hua, On an exponential sum, {\em J. Chinese Math. 
  Soc.}, 2 (1940), 301--312.   
  \bibitem{hua}L.-K. Hua, On exponential sums, {\em Sci. Record
    (Peking) (N.S.)}, 1 (1957), 1--4. 
\bibitem{Weyl} H. Weyl, \"{U}ber die Gleichverteilung der Zahlen
  mod. Eins, {\em Math. Ann.}, 77 (1916), 313--352.
\bibitem{WuXi} J. Wu and P. Xi, Arithmetic exponent pairs for
  algebraic trace functions and applications,
{\em Algebra Number Theory}, 15 (2021), no. 9, 2123--2172.
\end{thebibliography}
\end{document}